\documentclass[12pt,reqno]{amsart}
\usepackage{amsthm,amssymb,mathabx}
\usepackage{bm,fullpage}

\usepackage{tikz}
\usetikzlibrary{shapes,positioning,intersections,quotes}
\usepackage[font=small,labelfont=bf]{caption}

\usepackage[
  bookmarks = true, 
  pagebackref = false, 
  pdfauthor = {Anderson Palsson}, 
  pdftitle = \text{AP}, 
  colorlinks = true, 
  linkcolor = blue, 
  citecolor = blue]
{hyperref}

\newtheorem{thm}{Theorem}

\newtheorem{cor}{Corollary}

\theoremstyle{remark}

\newtheorem{remark}{Remark}

\newcommand{\R}{\ensuremath{\mathbb{R}}}

\newcommand{\Z}{\ensuremath{\mathbb{Z}}}

\newcommand{\N}{\ensuremath{\mathbb{N}}}

\newcommand{\eq}{\begin{equation}}
\newcommand{\ee}{\end{equation}}

\numberwithin{equation}{section}

\title{Bounds for discrete multilinear spherical maximal functions in higher dimensions}
\author[T. Anderson]{Theresa C. Anderson}
\address{
	Department of Mathematics
	\\ Purdue University
	\\ 150 N. University St.
	\\ W. Lafayette, IN 47907
	\\	U.S.A.
}
\email{tcanderson@math.purdue.edu}

\author[E. Palsson]{Eyvindur Ari Palsson}
\address{
	Department of Mathematics
	\\	Virginia Tech 
	\\	225 Stanger St.
	\\	Blacksburg, VA 24061
	\\	U.S.A.
}
\email{palsson@vt.edu}

\begin{document}
\maketitle

\begin{abstract}
We find the sharp range for boundedness of the discrete bilinear spherical maximal function for dimensions $d \geq 5$.  That is, we show that this operator is bounded on $l^{p}(\Z^d)\times l^{q}(\Z^d) \to l^{r}(\Z^d)$ for $\frac{1}{p} + \frac{1}{q} \geq \frac{1}{r}$ and $r>\frac{d}{2d-2}$ and we show this range is sharp.  Our approach mirrors that used by Jeong and Lee in the continuous setting.  For dimensions $d=3,4$, our previous work, which used different techniques, still gives the best known bounds.  We also prove analogous results for higher degree $k$, $\ell$-linear operators.
\end{abstract}

\section{Introduction}

In the work \cite{AP}, the authors used the Circle Method from analytic number theory to show bounds for the discrete bilinear  spherical maximal functions on the range $l^{p}(\Z^d)\times l^{q}(\Z^d) \to l^{r}(\Z^d)$ for $\frac{1}{p} + \frac{1}{q} \geq \frac{1}{r}$ and $r>\frac{d}{d-2}$ for all dimensions $d \geq 3$, as well as analogous multilinear results.  In this followup work, we show that the argument used by Jeong and Lee to show the sharp range for the continuous bilinear spherical maximal function \cite{JL} can be adapted to the discrete setting for dimensions $d\geq 5$; moreover, through testing specific functions, we show that this range is sharp.  Therefore, as in the linear setting, the discrete bounds are truly different from the continuous ones.  We emphasize that this technique does not appear to work in dimensions $3$ and $4$; therefore the bounds in \cite{AP} are still the best known for these cases.  An interesting open question would thus be to see if one can determine the sharp bounds for these two cases.  

We also have results for the multilinear and higher degree bounds.  The multilinear bounds use an induction argument, and these bounds are also sharp.  The higher degree bounds use information from \cite{ACHK} on the current best bounds for the linear degree $k$ spherical operator.  We state these results first for the $\ell$-linear, degree $2$ setting since these are easier to state and are also sharp.  We then state the bounds for the $\ell$-linear, degree $k$ version (which actually yields the previous case as a corollary).  The linear degree $k$ spherical maximal function is conjectured to be bounded on the sharp range $p > \frac{d}{d-k}$; if this was true then our bounde would also be sharp.  Therefore any improvements to bounds in the linear degree $k$ setting directly improve the $\ell$-linear degree $k$ results.  We comment that we have a similar loss in dimensionality when compared with \cite{AP}; that is, for the $\ell$-linear (in $\Z^{ld}$), degree $2$ setting, the observation that we show requires that $ d\geq 5$ compared with $d > 4/\ell$ in \cite{AP}.

All the precise statements and proofs appear in the next section.  We begin with the sharp bounds for the bilinear operator, then move to the multlinear and we close with an account of the higher degree version. We refer the reader to our previous work \cite{AP} and the references therein for more background and history on the topic.

We remark that the observation that the technique of Jeong and Lee carries over in the discrete setting was independently discovered by Michael Lacey and communicated to us \cite{ML}.  Due to the interest in the problem, we decided to make our observations public; we are interested in addressing applications and extensions in future work.

\subsection{Acknowledgements}
T. C. Anderson was supported in part by NSF DMS-1502464.  E. A. Palsson was supported in part by Simons Foundation Grant \#360560.

\section{Statements and proofs}
Recall that the discrete bilinear spherical averages can be written as
\[
T_\lambda(f,g)(\bm{x}) = \left| \frac{1}{N(\lambda)}\sum_{\bm{u}^2+\bm{v}^2=\lambda}f(\bm{x}-\bm{u})g(\bm{x}-\bm{v}) \right|
\]
where $\bm{u}$, $\bm{v} \in \Z^d$ and $N(\lambda) = \#\{ (\bm{u},\bm{v}) \in \Z^d\times\Z^d: |\bm{u}|^2+|\bm{v}|^2=\lambda\}$ is the number of lattice points on the sphere of radius $\lambda^{1/2}$ in $\R^{2d}$, which is $\lambda^{d-1}$ by the Hardy-Littlewood asymptotic for all $d \geq 3$.
We define the corresponding discrete bilinear spherical maximal operator as
\[\
T^*(f,g)(\bm{x}) = \sup_{\lambda \in \N} | T_\lambda(f,g)(\bm{x}) | 
\]

We will prove the following:
\begin{thm}
\label{Main theorem}
$T^*$ is bounded $l^{p}(\Z^d)\times l^{q}(\Z^d) \to l^{r}(\Z^d)$ for all $d \geq 5$ where $\frac{1}{p}+\frac{1}{q}\geq \frac{1}{r}$, $r >\frac{d}{2d-2}$ and $p,q>1$.  Moreover this range is sharp.
\end{thm}

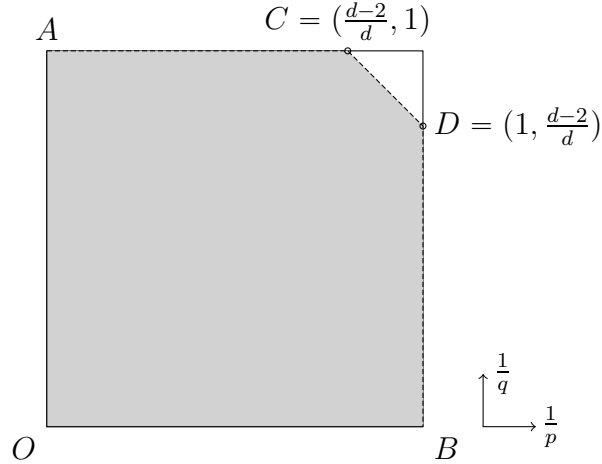
\begin{figure}[h!]\label{range}
\begin{tikzpicture}
\draw (0,0) rectangle (5,5);
\path[fill=gray!35] (0,0)--(0,5)--(4,5) --(5,4)--(5,0);
\draw[dash pattern= { on 2pt off 1pt}]  (0,5)--(4,5)--(5,4)--(5,0);
\draw (0,5)--(0,0)--(5,0);
\node [below left] at (0,0) {$O$};\node [above] at (0,5) {$A$};
\node [above] at (4,5) {$C=(\frac{d-2}{d},1)$};
\draw (4,5) circle [radius=0.04];
\node [right] at (5,4) {$D=(1,\frac{d-2}{d})$};
\draw (5,4) circle [radius=0.04];
\node [below right] at (5,0) {$B$};
\draw [<->] (5.8,0.7)--(5.8,0)--(6.5,0);
\node at (6.7,0) {$\frac{1}{p}$};
\node at (5.8,0.7) [right]{$\frac{1}{q}$};
\end{tikzpicture}
\caption{Range of $(\frac{1}{p},\frac{1}{q})$ for which $T^{*}$ is bounded on $l^p(\Z^d) \times l^q(\Z^d) \to l^r(\Z^d)$, $d\ge 5$.}
\end{figure}

\begin{proof}
We start with the first claim.  Jeong and Lee cleverly decomposed the continuous operator into a pointwise product of a Hardy-Littlewood maximal function and a (sub)-linear spherical maximal function.  Here we employ the same technique, which manifests itself nicely in the discrete setting.
We have
\[
|T_\lambda (f,g)(\bm{x})| \leq \frac{1}{\lambda^{\frac{d}{2}}}\sum_{|\bm{u}|^2 \leq \lambda}|f(\bm{x}-\bm{u})|\cdot \frac{1}{\lambda^{\frac{d}{2}-1}} \bigg|\sum_{|\bm{v}|^2 = \lambda - |\bm{u}|^2} g(\bm{x}-\bm{v})\bigg|.
\]
Next notice that
\[
\sup_{\lambda\in\N}\frac{1}{\lambda^{\frac{d}{2}-1}} \bigg|\sum_{|\bm{v}|^2 = \lambda - |\bm{u}|^2} g(\bm{x}-\bm{v})\bigg| \leq \sup_{\mu \in\N}\frac{1}{\mu^{\frac{d}{2}-1}} \bigg|\sum_{|\bm{v}|^2 = \mu} g(\bm{x}-\bm{v}) \bigg| = S(g)(\bm{x})
\]
where $S$ is the (discrete) linear spherical maximal function.   Therefore 
\[
|\sup_{\lambda\in\N}|T_\lambda (f,g)(\bm{x})| \leq \sup_{\lambda\in\N} \frac{1}{\lambda^{\frac{d}{2}}}\sum_{|\bm{u}|^2 \leq \lambda} |f(\bm{x}-\bm{u})| \cdot S(g)(\bm{x}) = M(f)(\bm{x}) \cdot S(g)(\bm{x})
\]
where $M$ is the usual (discrete) Hardy-Littlewood maximal function.  Now since $M$ is bounded on $l^p(\Z^d)$ for all $p>1$ and $S$ is bounded on $l^q(\Z^d)$ for all $q > \frac{d}{d-2}$, by H\"older's inequality and the nesting properties of the discrete $L^p$ spaces, we get $l^{p}(\Z^d)\times l^{q}(\Z^d) \to l^{r}(\Z^d)$ bounds for all $p>1$, $q >\frac{d}{d-2}$ and $r > \frac{d}{2d-2}$ where $\frac{1}{p}+\frac{1}{q}\geq\frac{1}{r}$.  Note that this argument only works for $d \geq 5$ since it depends on the Hardy-Littlewood asymptotic in $\Z^d$ (not $\Z^{2d})$ and the boundedness of the corresponding linear spherical maximal function.

Now we show that $r > \frac{d}{2d-2}$ is the sharp range.  Choose $f = \chi_L$ where $L = [-L, L]^d$ for some $L \geq 1$, and $g = \delta_0$.  We have that $f \in l^p(\Z^d)$ and $g\in l^q(\Z^d)$.  Then 
\[
\|\sup_{\lambda\in\N}|T_\lambda (f,g)\|_{l^r(\Z^d)} = \bigg(\sum_{\bm{x} \in \Z^d}\sup_{\lambda\in\N}\bigg|\frac{1}{\lambda^{d-1}}\sum_{|\bm{u}|^2+|\bm{v}|^2 = \lambda}\chi_L(\bm{x}-\bm{u})\delta_0(\bm{x}-\bm{v})\bigg|^r\bigg)^{1/r}
\]
\[
= \bigg(\sum_{\bm{x} \in \Z^d}\sup_{\lambda\in\N}\bigg|\frac{1}{\lambda^{d-1}}\sum_{\substack{|\bm{u}|^2=\lambda -|\bm{x}|^2\\ \bm{u} \in \bm{x} - L }} 1 \bigg|^r\bigg)^{1/r}
\]
by choosing $\lambda = 2|\bm{x}|^2$, so that $|\bm{u}|^2 = |\bm{x}|^2$ and $ \bm{u} \in \bm{x}-L$, the above is
\[
 \geq C\bigg(\sum_{\bm{x}\in\Z^d} \frac{1}{|\bm{x}|^{2(d-1)r}}\bigg)^{1/r} 
\]
which converges if and only if $r > \frac{d}{2d-2}$.  This argument relies on the fact that the sphere $|\bm{u}|^2 = |\bm{x}|^2$ has at least one lattice point, due to Lagrange's three and four square theorem, this necessary condition carries over for $d = 3,4$. 
\end{proof}

\begin{remark}
This result can easily be upgraded to weak type restricted estimates at endpoints using known endpoint estimates for the Hardy-Littlewood maximal function and endpoint estimates for the spherical maximal function. For the latter see the works of Ionescu \cite{Ionescu} and Hughes \cite{Hughes}. Similar comments also apply to all results stated below.
\end{remark}

We now comment on multilinear results.  We first state the result in the degree $2$ setting due to simplicity and to give an idea of the procedure.  We then extend it to the more general degree $k$ setting.
\begin{cor}
 $T^*(f_1, \dots , f_{\ell})$ is bounded on $l^{p_1}(\Z^d)\times\ldots\times l^{p_\ell}(\Z^d) \to l^{r}(\Z^d)$, $\frac{1}{p_1} + \ldots + \frac{1}{p_l} \geq \frac{1}{r}$, $r>\frac{d}{\ell d-2}$, $p_1,\ldots,p_{\ell}> 1$ and $d \geq 5$.
\end{cor}
\begin{remark}
Again, this is worse dimensionally than the dependence of $d > 4/\ell$ obtained in \cite{AP} (albeit for a suboptimal range of exponents).  When $\ell \geq 5$ and $d =1$, see also \cite{Cook} for a related estimate.  
\end{remark}

\begin{proof}
In the $l$-linear setting, the count $N(\lambda)$ is approximately $\lambda^{\frac{\ell d}{2}-1}$ as long as $d> 4/\ell$.  We can apply the same argument as in Theorem \ref{Main theorem} to bound $T^*$ by a pointwise product of $M$ times the $(\ell - 1)$-linear spherical maximal function.  We argue by induction that this is bounded on $l^{p_1}(\Z^d)\times\ldots\times l^{p_l}(\Z^d) \to l^{r}(\Z^d)$ where $\frac{1}{p_1} + \ldots + \frac{1}{p_l} \geq \frac{1}{r}$, $p_1,\ldots,p_{\ell}> 1$, $r>\frac{d}{\ell d-2}$ and $d \geq 5$, 
\end{proof}

To comment on sharpness in the $\ell$-linear setting, we extend our results to degree $k$ spheres, $k \geq 3$, $k \in \Z$ (if $k$ is odd, we assume that $y^k = |y|^k$).
Recall that the $\ell$-linear degree $k$ discrete spherical averages are defined as:
\[
T_\lambda(f_1,\ldots,f_{\ell})(\bm{x}) = \left| \frac{1}{N(\lambda)}\sum_{\bm{u_1}^k+\cdots + \bm{u_\ell}^k=\lambda}f_1(\bm{x}-\bm{u_1})\cdots f_\ell (\bm{x}-\bm{u_\ell}) \right|
\]
where $N(\lambda)$ is asymptotic to $\lambda^{\frac{\ell d}{k}-1}$ for all $d > d_0(k)/\ell$.  The best known values of $d_0(k)$, to the best of our knowledge, are contained in \cite{ACHK}.  For example, for $k=3$, we would need $d > 10/3$.  We define the maximal function $T^*$ analogously.  

Also define \[r_0(d,k) = \frac{2+2\delta_0(d,k)}{(\ell - 1)(2+2\delta_0(d,k)) + (1+2\delta_0(d,k))},\] where $\delta_0(d,k)$, defined on page 2 of \cite{ACHK}, relates to the best known bounds for the discrete linear degree $k$ operator.  Finally define \[p_0(d,k) = max \{ 1+\frac{1}{1+2\delta_0(d,k)}, \frac{d}{d-k}\}.\]  This $p_0(d,k)$ provides the best known $l^p$ bounds for the discrete linear operator, and is conjectured to be equal to $\frac{d}{d-k}$ for all $k \geq 3$.  See \cite{ACHK} for more details. We will prove:
\begin{thm}
 $T^*(f_1, \dots , f_{\ell})$ is bounded on $l^{p_1}(\Z^d)\times\ldots\times l^{p_l}(\Z^d) \to l^{r}(\Z^d)$, $\frac{1}{p_1} + \ldots + \frac{1}{p_l} \geq \frac{1}{r}$, $r>\max\{ r_0(d,k), \frac{d}{\ell d-k}\}$, $p_1,\ldots,p_{\ell}> 1$ and $d > d_{0}(k)$.  Moreover, the bound $r >  \frac{d}{\ell d-k}$ is a necessary condition.
\end{thm}
\begin{proof}
We begin with the first claim.  Using the Hardy-Littlewood asymptotic, we get that 
\[
|T_\lambda (f_1,\ldots,f_{\ell})(\bm{x})| \leq \frac{1}{\lambda^{\frac{d}{k}}}\sum_{|\bm{u_1}|^k \leq \lambda}|f_1(\bm{x}-\bm{u_1})|\cdot \frac{1}{\lambda^{\frac{(\ell -1)d}{k}-1}} \bigg|\sum_{|\bm{u_2}|^k+ \dots + |\bm{u_\ell}|^k = \lambda - |\bm{u_1}|^k} f_2(\bm{x}-\bm{u_2})\cdots f_\ell (\bm{x} - \bm{u_\ell})\bigg|.
\]
So 
\[
\sup_{\lambda\in\N}|T_\lambda (f,g)(\bm{x})| \leq M_k(|f_1|) \cdot S_k^{\ell - 1}(f_2, \dots , f_\ell )
\]
where $M_k$ is the discrete Hardy-Littlewood maximal function defined over $k$-balls and $S_k^{\ell - 1}$ is the discrete $\ell-1$ linear $k$ spherical maximal function; this is true for any re-arrangement of the functions.  Again, using induction, we see that $S_k^{\ell - 1}$ is bounded on $l^{p_2}(\Z^d)\times\ldots\times l^{p_l}(\Z^d)$, for all $p_2, \dots p_{\ell-1} > 1$, $p_\ell > p_0(d,k)$, $d > d_0(k)$, while $M_k$ is bounded for all $p_1 >1$.  By H\"older's inequality, we get the range of $r$. 

If we knew that the discrete linear $k$ spherical maximal function was bounded on the conjectured sharp range of $p_0(d,k) = \frac{d}{d-k}$, then we would have $r > \frac{d}{\ell d-k}$, which is the necessary condition that we claim.

Now we show this necessary condition.  We modify the earlier example.  Choose $f_1 = \chi_L$ where $L = [-L, L]^d$ for some $L \geq 1$, and $f_2, \dots, f_\ell = \delta_0$.  Then 
\[
\|\sup_{\lambda\in\N}|T_\lambda (f_1, \dots , f_\ell))\|_{l^r(\Z^d)} = 
\bigg(\sum_{\bm{x} \in \Z^d}\sup_{\lambda\in\N}\bigg|\frac{1}{\lambda^{\frac{\ell d}{k}-1}}\sum_{|\bm{u_1}|^k+\cdots +|\bm{u_\ell}|^k =  \lambda}\chi_L(\bm{x}-\bm{u_1})\delta_0(\bm{x}-\bm{u_2})\dots\delta_0(\bm{x}-\bm{u_\ell})\bigg|^r\bigg)^{1/r}
\]
\[
= \bigg(\sum_{\bm{x} \in \Z^d}\sup_{\lambda\in\N}\bigg|\frac{1}{\lambda^{\frac{\ell d}{k}-1}}\sum_{\substack{|\bm{u_1}|^k=\lambda -(\ell -1)|\bm{x}|^k\\  \bm{u} \in \bm{x}-L }}\bigg|^r\bigg)^{1/r}
\]
by choosing $\lambda = \ell|\bm{x}|^k$, so that $|\bm{u}|^k = |\bm{x}|^k$ and $\bm{u} \in \bm{x}-L$ the above is
\[
 \geq C\bigg(\sum_{\bm{x}\in\Z^d} \frac{1}{|\bm{x}|^{(\ell d -k))r}}\bigg)^{1/r} 
\]
which converges if and only if $r \geq \frac{d}{\ell d-k}$, which matches the sharp range for $k=2$.  Since this proof of the necessary condition does not rely on the bounds for the discrete degree $k$ spherical maximal function, we can actually dip lower dimensionally than $d_0(k)$ here.  In particular, here we can have $d \geq 4k$ when $k \geq 4$ is a power of 2 and $d \geq \frac{3}{2}k$ otherwise \cite{Vaughan}. 
\end{proof}

\bibliographystyle{amsplain}

\begin{thebibliography}{99}

\bibitem{ACHK}
T. Anderson, B. Cook, K. Hughes, and A. Kumchev,  \emph{Improved $l^p$ boundedness for Integral k-Spherical Maximal Functions}.  Discrete Analysis, May 29, 2018.

\bibitem{AP}
T. C. Anderson and E. A. Palsson, \emph{Bounds for discrete multilinear spherical maximal functions}.  Submitted.  Preprint on arXiv. 

\bibitem{Cook}
B. Cook, \emph{Discrete multilinear spherical averages}.  To appear in Canad. Math. Bull.

\bibitem{Hughes}
K. Hughes, \emph{Restricted weak-type endpoint estimates for $k$-spherical maximal functions}, Math. Z. \textbf{286} (2017), no. 3-4, 1303-1321.

\bibitem{Ionescu}
A. D. Ionescu, \emph{An endpoint estimate for the discrete spherical maximal functions}, Proc. Am. Math. Soc. \textbf{132} (2004), no. 5, 1411-1417.

\bibitem{JL}
Jeong and Lee, \emph{Maximal estimates for the bilinear spherical averages and the bilinear Bochner-Riesz operators}.  Preprint on arXiv.

\bibitem{ML}
Michael Lacey. Personal communication, 2019.

\bibitem{Vaughan}
R. C. Vaughn, \emph{The {H}ardy--{L}ittlewood {M}ethod}, {S}econd ed.,
  Cambridge University Press, 1997.
  
\end{thebibliography}

\end{document}